\documentclass{article}
\usepackage[utf8]{inputenc}
\usepackage{graphicx}
\usepackage{float}
\usepackage{amsmath}
\usepackage{stmaryrd}
\usepackage{amsthm}
\usepackage{amssymb}
\usepackage{enumitem}
\usepackage{geometry}
\usepackage{nccmath}
\usepackage[english]{babel}
\theoremstyle{definition}
\newcommand\ddfrac[2]{\frac{\displaystyle #1}{\displaystyle #2}}
\newcommand{\pwr}[1]{2^{2^{#1}}}
\newcommand{\wir}{\textrm{wir}}

\newcommand{\volx}[1]{(\pwr{2(#1)}f(#1)+1)f(#1)}
\newcommand{\voly}[1]{(\pwr{2#1+1}f(#1)+1)f(#1)}
\newtheorem{theorem}{Theorem}[section]
\newtheorem{definition}[theorem]{Definition}
\newtheorem{lemma}[theorem]{Lemma}

\title{Dependence of the coarse wiring profile on the choice of parameter}
\author{Ruth Raistrick}
\date{\today}

\begin{document}

\maketitle

\begin{abstract}
There are two subgraphs $X,Y$ of the $2$-dimensional integer grid such that for any $k<l$ there is an infinite subset $I\subset \mathbb{N}$ such that the $k$-wiring profile of $X$ into $Y$ grows quadratically on $I$ while the $l$-wiring profile of $X$ into $Y$ grows linearly on $I$. This resolves Question 7.8 of Barrett-Hume.
\end{abstract}


\section{Introduction}
In \cite{Barrett-Hume}, the authors introduce coarse wirings as a coarse geometric version of thick embeddings defined by Kolmogorov-Barzdin and later considered by Gromov-Guth \cite{KB,GG}.

\begin{definition}\label{defn:kwiring} Let $\Gamma$ be a finite graph and let $Y$ be a graph. A \textbf{wiring} of $\Gamma$ into $Y$ is a continuous map $f:\Gamma\to Y$ which maps vertices to vertices and edges to unions of edges.

A wiring $f$ is a \textbf{coarse $k$-wiring} if
\begin{enumerate}
 \item the restriction of $f$ to $V\Gamma$ is $\leq k$-to-$1$, i.e.\ $|\{v\in V\Gamma\mid f(v)=w\}|\leq k$ for all $w\in VY$; and
 \item each edge $e\in EY$ is contained in at most $k$ of the paths in $\mathcal P$.
\end{enumerate}
We consider the \textbf{image} of a wiring $\textrm{im}(f)$ to be the subgraph of $Y$ consisting of all vertices and edges in the image of $f$. The \textbf{volume} of a wiring $\textrm{vol}(f)$ is the number of vertices in its image.  We denote by $\wir^k(\Gamma\to Y)$ the minimal volume of a coarse $k$-wiring of $\Gamma$ into $Y$. If no such coarse $k$-wiring exists, we say $\wir^k(\Gamma\to Y)=+\infty$.

Let $X$ and $Y$ be graphs. The $k$-\textbf{wiring profile} of $X$ into $Y$ is the function
\[
 \wir^k_{X\to Y}(n) = \max\{\wir^k(\Gamma\to Y)\mid\Gamma\leq X,\ |\Gamma|\leq n\}.
\]
\end{definition}

For the graphs considered in \cite{Barrett-Hume} $\wir^k_{X\to Y}(n)$ does not depend (up to a natural asymptotic equivalence of functions) on the parameter $k$, provided it is chosen to be sufficiently large. We recall that given two functions $f,g:\mathbb{N}\to\mathbb{N}$ we write $f\lesssim g$ if there exists some constant $C$ such that $f(n)\leq Cg(Cn)+C$ holds for every $n$. We write $f\simeq g$ if $f\lesssim g$ and $g\lesssim f$.

The authors ask whether there exist pairs of graphs where for every $k$, there is some $l>k$ such that $\wir^k_{X\to Y}(n)\not\simeq\wir^l_{X\to Y}(n)$ \cite[Question 7.8]{Barrett-Hume}. The goal of this paper is to prove that there are.

\begin{theorem}\label{thm:main} There are two subgraphs $X,Y$ of the $2$-dimensional integer grid such that, for every $k\geq 2$ there is an infinite subset $I_k\subset \mathbb{N}$ such that for all $n\in I_k$
\begin{itemize}
	\item[(a)] $\wir^{k-1}_{X\to Y}(n) \geq \left(\frac{n}{k(k+1)}\right)^2$, and
	\item[(b)] $\wir^k_{X\to Y}(n) \leq 2n$.
\end{itemize}
\end{theorem}

The remainder of the paper is split into three short sections. The first contains the definitions of the graphs $X$ and $Y$, the second contains the proof of part (a) of the main theorem, while the third contains the proof of part (b).

\subsection*{Acknowledgements}
This research was conducted during a summer research programme funded by the University of Bristol under the supervision of Dr David Hume.

%

\section{Defining the graphs $X$ and $Y$}
The following two graphs will be crucial for the proof of Theorem \ref{thm:main} as they will be the examples as needed.
\begin{definition}
Define $f:\mathbb{N} \to \mathbb{N}$ as a surjective function such that $\forall n \in \mathbb{N}\backslash \{1\},$ $|f^{-1}(n)|=\infty$ and $2\leq f(n) \leq n.$ $f(1)=1.$
\end{definition}
For example $f(n)$ (for $n\geq 2$) could be the enumeration of the smallest prime in the prime decomposition of $n$ plus $1$ where the primes are enumerated by their order. i.e. $2=p_1,$ $3=p_2,$ $\ldots$. 
\begin{definition}
We define the graph $X_n=(VX_n,EX_n)$ to be a graph with $\volx{n}$ vertices where $VX_n=\{0_n,1_n,2_n, \ldots, f(n)-1_n \} \times \{0_n,1_n,2_n, \ldots, 2^{2^{2n}}f(n)_n\}$ and 
\begin{multline*}
    EX_n = \{ e | e=(i_n, j_n)(i_n,j+1_n) \: 0 \leq j \leq 2^{2^{2n}}f(n)-1\} \cup \\ \{ e | e=(i_n,j_n)(i+1_n,j_n) \: j = \pwr{2n}m, \text{ for some } m \in \mathbb{N}\cup \{0\}  \}.
\end{multline*}
We define $X$ to be the graph whose components are these graphs. i.e. $$X=\bigcup_{n=1}^\infty X_n$$
\end{definition}
We now define $Y$ to be a very similar graph.
\begin{definition}
We define the graph $Y_n=(VY_n,EY_n)$ to be a graph with $\voly{n})$ vertices where $VY_n=\{0_{n_y},1_{n_y},2_{n_y}, \ldots, f(n)-1_{n_y} \} \times \{0_{n_y},1_{n_y},2_{n_y}, \ldots, 2^{2^{2n+1}}f(n)_{n_y}\}$ where we drop the second subscript where it is explicitly clear we are working in  some $Y_n$, and  
\begin{multline*}
    EY_n = \{ e | e=(i_n, j_n)(i_n,j+1_n) \: 0 \leq j \leq 2^{2^{2n+1}}f(n)-1\} \cup \\ \{ e | e=(i_n,j_n)(i+1_n,j_n) \: j = (\pwr{2n+1}m) \text{ for some } m \in \mathbb{N}\cup \{0\}  \}.
\end{multline*}
We define $Y$ to be the graph whose components are these graphs. i.e. $$Y=\bigcup_{n=1}^\infty Y_n$$
\end{definition}

The set $I_k$ in the proof of Theorem \ref{thm:main} will be $\{\volx{n}\mid f(n)=k\}$.

\section{Lower bounds for $(f(n)-1)$-wirings}
We now provide a proof of part (a) of Theorem \ref{thm:main}.
%

\begin{lemma}
Where $f(n)\geq 2$ (equivalently $n\neq1$) any $(f(n)-1)$-wiring of $X_n$ into $Y$ has volume at least $2\cdot2^{2^{2n+1}}+1.$ Hence,
\[
\wir^{f(n)-1}_{X\to Y}(\volx{n}) \geq 2^{2^{2n+1}} \geq \left(\frac{\volx{n}}{f(n)(f(n)+1)}\right)^2.
\]
\end{lemma}
\begin{proof}
Fix $n \in \mathbb{N}$. Suppose, for the sake of contradiction that there exists a $(f(n)-1)$-wiring of $X_n$ into $Y$ with volume less than $2\cdot 2^{2^{2n+1}}+1$, call this $g.$ It is clear that as $X_n$ is connected $g$ must send it to one connected component of $Y$, i.e. $Y_m$ for some natural $m$. 
\\
We first aim to show that $m\geq n.$ Suppose $m<n,$ as they're both integers $2m+1<2n$ which gives that 
$$2^{2^{2m+1}}<2^{2^{2n}}.$$
So as $|Y_m| = (2^{2^{2m+1}}f(m)+1)f(m)$ by the pigeonhole principle there exists $v \in EY_m$ such that 
\begin{align*}
    |g^{-1}(v)| &\geq \ddfrac{(2^{2^{2n}}f(n)+1)f(n)}{(2^{2^{2m+1}}f(m)+1)f(m)} 
\end{align*}
So to show that $g$ is not a $(f(n)-1)-$wiring it suffices to show
$$\ddfrac{(2^{2^{2n}}f(n)+1)f(n)}{(2^{2^{2m+1}}f(m)+1)f(m)} > f(n)-1$$ and as $$\ddfrac{(2^{2^{2n}}f(n)+1)f(n)}{(2^{2^{2m+1}}f(m)+1)f(m)} >\ddfrac{\pwr{2n}f(n)^2}{\pwr{2m+1}f(m)^2}$$ to show this it suffices to show 
$$\ddfrac{\pwr{2n}f(n)^2}{\pwr{2m+1}f(m)^2}>f(n)-1.$$
Supposing the contrary gives
\begin{gather*}
    2^{2^{2n}-2^{2m+1}}\mfrac{f(n)^2}{f(m)^2} \leq f(n)-1 \\
    \implies 2^{2^{2n}-2^{2m+1}} f(n)^2 \leq f(m)^2(f(n)-1) \\
    \implies 2^{2^{2n}-2^{2m+1}} f(n) \leq f(m)^2 - \mfrac{f(m)^2}{f(n)}
\end{gather*}
It is not hard to see that $f(m)^2 - \mfrac{f(m)^2}{f(n)}$ is maximised when $f(m)^2$ is and when $f(n)$ is, so as $f(k)\leq k$ $\forall k \in \mathbb{N},$ $f(m)^2 - \mfrac{f(m)^2}{f(n)}\leq m^2(1-\frac{1}{n}).$ We also have that $f(n)\geq 1.$ Combining these facts we get:
\begin{gather*}
    2^{2^{2n}-2^{2m+1}} \leq 2^{2^{2n}-2^{2m+1}} f(n) \leq f(m)^2(1 - \frac{1}{f(n)}) \leq m^2(1-\frac{1}{n}) \leq m^2 \\
    \implies 2^{2^{2n}-2^{2m+1}} \leq m^2 
\end{gather*}
Now $$2^{2^{2n}-2^{2m+1}}=2^{2^{2m+1}(2^{2(n-m)-1}-1)}$$
and as argued previously, as $n,m \in \mathbb{N}$ and $m<n,$ $2(n-m)-1\geq 1$ hence $2^{2(n-m)-1}-1\geq 2^1-1=1.$ So
$$m^2 \geq 2^{2^{2n}-2^{2m+1}}\geq 2^{2^{2m+1}}$$
But clearly $\forall m \in \mathbb{N}$ $m^2 < 2^{2^{2m}},$ hence we have a contradiction and so $m\geq n.$
\\
So let $m\geq n$. The smallest cycle in $Y_m$ is of volume $2\cdot \pwr{2m+1}+2.$ (This is as for a cycle we need two complete `vertical' paths -  abusing notation paths of the form $(i_m, k\pwr{2m+1}_m),(i_m, k\pwr{2m+1}+1_m), \ldots(i_m,(k+1)\pwr{2m+1}_m)$ - and the least `distance' between any two is only 1 edge at the `top' and `bottom' so the two lots of $\pwr{2m+1}+1$ on the vertical paths and no more vertices give $2\cdot \pwr{2m+1}+2$ vertices). And $\pwr{2m+1}\geq \pwr{2n+1}$ $\implies$ $2\cdot \pwr{2m+1}+2>2\cdot \pwr{2n+1}$. Therefore $g(X_n)$ cannot contain a cycle, hence it is acyclic and as previously stated connected, so $g(X_n)$ is a tree. 
Now consider the `bottom' row of $X_n$ (i.e. the vertices of the form $(i_n,0_n)$) as we have an $f(n)-1$ wiring and there are $f(n)$ of these there must exist $v_0,v_0'$ of this form such that $g(v_0)\neq g(v'_0)$ and $v_0v'_0\in EX_n$ (i.e. wlog $v'_0=v_0+\mathbf{e_1}).$ 

To see that this is true suppose no such $v_0, v'_0$ exist, this implies all pairs vertices one apart horizontally in $X_n$ are mapped to the same point under $g,$ but all vertices along a horizontal line are one apart from another vertex along the line so this immediately gives all of the form $(i_n,0_n)$ are sent to the same vertex under $g,$ a clear contradiction as there are $f(n)$ of them and $g$ is a $(f(n)-1$)- wiring.
\\ 
Let us replicate this for each `row' of $X_n$ so $\forall u \in \{0,1, \ldots, f(n)-1 \}$ we have $v_u,v'_u$ so $v'_u=v_u+\mathbf{e_1}$ and $g(v_u)\neq g(v'_u).$
So let's consider the image of the edge between $v_0$ and $v'_0$ under $h$. As $g(v)\neq g(v')$ $g(v_0v'_0)$ is a path or edge say $e_1,e_2,\ldots,e_p$ where the `starting' point of $e_1$ is $g(v_0)$ and ending point of $e_p$ is $g(v'_0).$ As $g(X_n)$ is acyclic this path is the only such linking $g(v_0)$ to $g(v'_0)$. In $X_n$ we can see there are $\geq f(n)+1$ distinct paths between $v_0$ and $v'_0$ (going through all $f(n)+1$ of the horizontals, i.e. through the path $p=(0_n,m\pwr{2n}_n),(1_n,m\pwr{2n}_n),\ldots,(f(n)-1_n,m\pwr{2n}_n)$ for all $m \in \{1,2,\ldots, f(n) \}$ and the most direct path between $v_0$ and $v'_0$ which consists only of edges with start and end point with $0_n$ as the `y-coordinate') all of which must be mapped to a path/ edge (but not single vertex as $g(v)\neq g(v')$) in $g(X_n)$ which links $g(v)$ to $g(v')$. But we've just seen there is only one such path so all $f(n)+1$ paths must be sent to the same path, namely $g(v_0v'_0).$ 

We also note that if we enumerate these distinct paths in $X_n$, the $0$th path being $v_0v'_0$, the $m$th being the one that passes through $(0_n,m\pwr{2n}_n),(1_n,m\pwr{2n}_n),\ldots,(f(n)-1_n,m\pwr{2n}_n)$ we have that $v_uv'_u$ is in the $u$th path for all $0\leq u \leq f(n).$ Hence $g(v_uv'_u)$ is in $g(v_0v'_0)$ for all such $u$ (i.e. the path $g(v_uv'_u)$ is entirely contained in the path $g(v_0v'_0)$). Now change 0 to 1 and using the identical argument get that $g(v_uv'_u)$ is in $g(v_1v'_1)$ for all $u$, in fact we replicate the argument for all $v_mv'_m$ and get that the path $g(v_uv'_u)$ is entirely contained in the path $g(v_mv'_m)$ for all $u,m \in \{0,1, \ldots, f(n)\}.$
This implies $g(v_0v'_0)=g(v_1v'_1)=\ldots=g(v_{f(n)-1}v'_{f(n)-1}).$ 

Now we've seen $e_1$ is in this path so taking $e_1$ we get $e_1 \in g(v_uv'_u)$ for all such $u,$ hence 
$$|g^{-1}(e_1)|\geq f(n)+1$$
a direct contradiction to $g$ being a $(f(n)-1)$-wiring.
\end{proof}

\section{Upper bounds for $f(n)$-wirings}
We finish by proving part (b) of Theorem \ref{thm:main}. We start with a simple observation.

\begin{lemma}\label{lemma-subdivision} There is a $1$-wiring of any subgraph of $X_n$ into $Y$ with volume at most $\voly{n}$.
\end{lemma}
\begin{proof}
This follows immediately from the fact that $Y_n$ is a subdivision of $X_n$.
\end{proof}

When considering $f(n)$-wirings we get a stronger upper bound.

\begin{lemma}
Every subgraph of $X$ with at most $\volx{n}$ vertices admits a $f(n)$-wiring in $Y$ with volume at most $2\volx{n},$ given $f(n)\geq2 \quad (\iff n \neq 1).$ Hence,
\[
 \wir^{f(n)}_{X\to Y}(\volx{n})\leq 2\volx{n}.
\]
\end{lemma}
\begin{proof}
We shall now define a wiring then show it is as required. So let me define $r:X \to Y$. Let $v \in VX$ it follows that there exists unique $n$ such that $v \in VX_n$ hence $v=(i_n,j_n)$ for some $i_n,j_n.$ We define $r$ on the vertices by
$$r(v) = r((i_n,j_n)) = (0_{n_y}, j_{n_y}).$$
Let $e=vw \in EX$ then, again, there exists unique $n$ such that $e \in EX_n$, so $e=(i_n,j_n)(i_n, j_n +1)$ or $e=(i_n,j_n)(i+1_n,j_n)$ for some $i_n,j_n$. We define $r$ on the edges as follows:
$$r(e)=r(vw) =r((i_n,j_n)w)=\bigg\{ \begin{array}{cc}
    (0_{n_y},j_{n_y}) & \text{ if }  w=(i+1_n,j_n)  \\
    r(v)r(w) & \text{ otherwise }
\end{array}$$
i.e. all `horizontal' edges are mapped to a point as the ends of all such edges have been mapped to the same point by $r$ and `vertical' edges are mapped to one edge in the only way possible.
To see this is a valid wiring note that if $e$ is a `horizontal' edge in $X_n$ then $e=vw$ where $v,w$ have the same `$y$-coordinate value hence are mapped to the same point, and so is the edge between them, by definition of the mappings of `vertical' edges they're valid. \\
It is not hard to see that for all $\Gamma $ where $\Gamma$ is a subgraph of $X$, $vol(r(\Gamma)) \leq vol(\Gamma)$ as no edge in $\Gamma$ is mapped to a path longer than one edge in $Y$ so we have no `added' extra vertices from paths, equivalently $\forall v \in Vr(\Gamma), \exists v' \in V\Gamma $ so $r(v')=v.$ 
\\
To see this is a $f(n)-$wiring we'll first see that if $\Gamma=X_n$ for some $n\in \mathbb{N}$ then we have an $f(n)$-wiring.
Let $v \in Vr(\Gamma)$ it follows that $v=(0_{n_y},j_{n_y})$ for some $j_{n_y}\in \{0_{n_y}, \ldots, \pwr{2n+1}f(n)-1_{n_y}\}.$ Then $$r^{-1}(v)=\{(x,j_{n}) \: | x \in \{0_n, \ldots, f(n)-1_n \}\}$$
so we see there are $f(n)$ options for $x$, hence $|r^{-1}(v)|=f(n)-1.$ As $v$ was chosen arbitrarily this holds for all such vertices. Now let $e\in Er(\Gamma)$, we want to show $\leq f(n)$ $e' \in E\Gamma$ such that $e \in r(e'),$ but as $r(e')$ is a single edge or vertex $\forall$ $e' \in E\Gamma$ we have that $e \in r(e') \iff e=r(e').$ So it suffices to prove $|r(e')| \leq f(n) \: \forall e' \in E\Gamma$ and only where $r(e') \in EY$ as we do not care how many edges are mapped to a vertex in the definition of a coarse $k-$wiring. So suppose $r(e') \in EY,$ ($r(e')$ fixed) then we have that $r(e') = (0_{n_y},j_{n_y})(0_{n_y},j+1_{n_y})$ for some $j_{n_y} \in \{0_{n_y},\ldots, \pwr{2n+1}-1_{n_y} \}$ (without loss of generality we can put the vertices in this order). It follows that if $e \in r(e'),$ $e$ is of the form $(i_n,j_n)(i_n,j+1_n)$ where $j_n$ and hence $j+1_n$ are fixed. So only $i_n$ is variable as we know $\forall v \in VX_n$ the first coordinate can only take $f(n)$ values, namely $0_n, \ldots, f(n)-1_n$, hence $|r(e')|\leq f(n).$ So where $\Gamma=X_n$ it is an $f(n)$-wiring. This immediately gives that if $\Gamma$ is a subgraph of $X_n$ for some $n\in \mathbb{N}$ which is the same $n$ such that $\volx{n}$ is the smallest such that $\volx{n} \geq vol(\Gamma),$ $r$ is a valid $f(n)-$wiring. And we see $$vol(r(\Gamma))=|Vr(\Gamma)|=|\{(0,j_{n_y})\: :\: 0\leq j_n \leq \pwr{2n}f(n)\}| = \pwr{2n}f(n)+1.$$ Hence $r$ is as required. 
\\
But what if $\Gamma$ is composed of many components subgraphs of many $f(n)$s, or even $\Gamma$ is the subgraph of one $X_n$ but this is not the same $n$ such that $\volx{n}$ is the smallest such that $vol(\Gamma) \leq \volx{n}?$ 
Suppose $\Gamma$ is as above, we split $\Gamma$ into $m\leq b$ subgraphs $\Gamma_1, \ldots, \Gamma_m$ such that $\Gamma_i$ is a subgraph of $X_{n_i}$ for some $n_i\in \mathbb{N}$ and there does not exist any other subgraph of $X_{n_i}$ in any other of the $\Gamma_js.$ For ease we enumerate them so that $n_1 < n_2 < \ldots < n_m.$ We fix $n$ to be as described just beneath the lemma, the smallest such that $vol(\Gamma)\leq \volx{n}.$ We now split into three scenarios: $n_m <n,$ $n_m = n,$ and $n_m>n.$

Let $n_m<n,$ with some though we can see that this implies $m<n$ as at it's maximum we have $n_1=1,n_2=2, \ldots, n_m=m<n.$ We also see $vol(\Gamma_i)\leq \volx{n_i}$ for all $i$ as $\Gamma_i$ is a subgraph of $X_i.$ Using the 1-wiring from Lemma \ref{lemma-subdivision} which we know gives $vol(h(\Gamma_i))\leq \voly{n_i}.$ Hence:
\begin{align*}
    vol(h(\Gamma))&=\sum_{i=1}^m vol(h(\Gamma_i)) \\
    &\leq \sum_{i=1}^m (\pwr{2n_i+1}f(n_i)+1)f(n_i) \\
    &\leq \sum_{i=1}^m (\pwr{2n_i+1}(n_i)+1)(n_i\\
    &\leq \sum_{i=1}^m (\pwr{2n_i+1}(n-1)+1)(n-1)\\
    &\leq \sum_{i=1}^m (\pwr{2n-1}(n-1)+1)(n-1) \\
    &\leq (\pwr{2n-1}(n-1)+1)(n-1)^2\\
    &\leq \pwr{2n} \leq \volx{n}
\end{align*}
As $h$ is a 1-wiring it must also be a $f(n)$-wiring so $h$ is as required in this case and we are done.

Now the case where $n_m= n$, like above we apply $h$ to $\Gamma_1,\Gamma_2,\ldots,\Gamma_{m-1}$ and we apply $r$ to $\Gamma_m,$ call this map $\zeta.$ Then 
\begin{align*}
    vol(\zeta(\Gamma))&=\sum_{i=1}^m vol(\zeta(\Gamma_i)) \\
    &= vol(r(\Gamma_m))+\sum_{i=1}^{m-1} vol(h(\Gamma_i))\\
    &\leq \pwr{n}f(n)+1 + \sum_{i=1}^{m-1} \voly{n_i}\\
    &\leq \pwr{n}f(n)+1 + \sum_{i=1}^{m-1} (\pwr{2(n-1)+1}(n-1)+1)(n-1)\\
    &\leq \pwr{n}f(n)+1 +(\pwr{2(n-1)+1}(n-1)+1)(n-1)^2\\
    &\leq \pwr{2n}f(n)+1 + \pwr{2n} \\
    &\leq \pwr{2n}f(n)^2 +1 \\
    &\leq \pwr{2n}f(n)^2 +f(n) =\volx{n}
\end{align*}
So $\zeta$ is as required.
 
Now let's consider $n_m>n.$  We simply use the wiring $\zeta$ from above on $\Gamma_1, \ldots, \Gamma_j$ where $n_j \leq m < n_{j+1}$ and the 1-wiring $\phi$ as described in the proof of lemma 2.5 on $\Gamma_{j+1},\ldots,\Gamma_m.$ Let's call this map $\zeta',$ then:
\begin{align*}
    vol(\zeta'(\Gamma))&=\sum_{i=1}^j vol(\zeta(\Gamma_i))+\sum_{i=j+1}^m vol(\phi(\Gamma_i)) \\
    &\leq \volx{n} + \sum_{i=j+1}^m vol(\Gamma_i)\\
    &\leq \volx{n}+\volx{n}\\
    &= 2\volx{n}
\end{align*}
Which is as required, so as we have covered all possible situations we are done.
\end{proof}

\newpage
\def\cprime{$'$}

\end{document}